\documentclass[12pt,reqno]{amsart}
\usepackage{amsmath,amssymb,amsthm,amsfonts,latexsym,mathrsfs,verbatim,comment}
\usepackage[hidelinks,unicode]{hyperref}

\newtheorem{theorem}{Theorem}[section]

\newtheorem{lemma}[theorem]{Lemma}
\newtheorem{proposition}[theorem]{Proposition}
\theoremstyle{remark}
\newtheorem{remark}[theorem]{\bf Remark}

\newcommand{\cA}{\mathcal A}
\newcommand{\cP}{\mathcal P}
\newcommand{\cC}{\mathcal C}
\newcommand{\LL}{\pounds}
\newcommand{\leg}[2]{\left(\frac{#1}{#2}\right)}
\DeclareMathOperator{\logA}{\log_{\mathcal{A}}}

\DeclareMathOperator{\ord}{ord}
\DeclareMathOperator{\rad}{rad}
\DeclareMathOperator{\Rec}{Rec}
\DeclareMathOperator{\im}{im}

\begin{document}
\title[Irrationality of finite logarithms]{Irrationality of finite logarithms \\in a congruence-class ad\`{e}le ring}
\author{Daniel Evans}
\address{Department of Mathematical Sciences, Durham University, DH1 3LE, United Kingdom}
\email{daniel.evans@durham.ac.uk}



\begin{abstract}
Finite logarithms can be defined in the ``poor man's ad\`{e}le ring" $\mathcal{A}$ using Fermat quotients modulo sufficiently large primes. This ring contains $\mathbb{Q}$ and outside the trivial cases, Matsusaka and Seki have shown that finite logarithms cannot take non-zero rational values in $\mathcal{A}$. Furthermore, a theorem of Silverman shows they are not zero, assuming the $abc$-conjecture. We extend these results to primes restricted to arithmetic progressions of the form $p\equiv 1\bmod m$ by relating Fermat quotients to values of cyclotomic polynomials and their logarithmic derivatives. A signed version of the same argument shows unconditionally that the square of a finite logarithm cannot take non-zero rational values in $\mathcal{A}$.  
As a further application, we show that, subject to the $abc$-conjecture, finite logarithms cannot be quadratic irrational elements of $\mathcal{A}$ in Rosen's theory of finite algebraic numbers.
\end{abstract}

\maketitle

\section{Introduction}
In recent years, there has been growing interest in arithmetic in the so-called {\it poor man's ad\`{e}le ring}
\[\cA:=\left.\left(\;\prod_{p\in\cP}\;\mathbb{Z}/p\mathbb{Z}\right) \middle/ \left(\;\bigoplus_{p\in\cP}\;\mathbb{Z}/p\mathbb{Z}\right)\right.. \]
The direct product and sum are taken over the set of primes $\cP$, and addition and multiplication are taken component-wise.
The field of rational numbers $\mathbb{Q}$ embeds naturally along the diagonal, and in this manner we can view $\mathcal{A}$ as a $\mathbb{Q}$-algebra. 
Explicitly, elements ${\bf t}\in\cA$ can be encoded as vectors
\[{\bf t}=(t_2, t_3, t_5,...) = (t_p)_{p\in\mathcal{P}}\quad\text{where $t_p\in\mathbb{Q}$}\]
and a finite number of reductions $t_p\bmod p$ may remain undefined when $t_p$ is not $p$-integral. 
One then identifies ${\bf t}={\bf s}$ when $t_p\equiv s_p\bmod p$ for all sufficiently large primes $p\in\cP$. Going further, Rosen (see \cite{Ro20} and \cite{RTTY24}) has developed a notion of algebraic numbers in $\mathcal{A}$, which are essentially ${\bf t}=(t_p)_{p\in\mathcal{P}}$ arising from sequences of rational numbers $t_n$ satisfying a linear recurrence with constant rational coefficients. 
Alongside many aspects concerning finite multiple zeta values, in \cite{KZ} Kaneko and Zagier discuss several other arithmetically interesting special elements in $\cA$, including Fermat quotients modulo $p$ as the prime $p$ varies. 
For prime $p$, write $\mathbb{Z}_{(p)}$ for the localisation of $\mathbb{Z}$ at the prime ideal $(p)$. 
Given $\alpha\in\mathbb{Z}_{(p)}^\times$, the {\it Fermat quotient} is
\[q_p(\alpha)=\frac{\alpha^{p-1}-1}{p}\in\mathbb{Z}_{(p)}\]
and we have a corresponding function
\begin{align*}
\logA:\mathbb{Q}^\times & \longrightarrow\cA \\
\alpha& \longmapsto \left(q_p(\alpha)\bmod p\right)_{p\in\cP}.
\end{align*}
It is easy to check $\log_\cA(\alpha\beta)=\log_\cA(\alpha)+\log_\cA(\beta)$, so this provides a natural analogue of the logarithm function with values in $\cA$. On investigating irrationality and transcendence in $\cA$, Luca and Zudilin \cite{LZ25} suggested an interesting step would be to prove $\log_\cA(2)$ is irrational. This was answered by Matsusaka and Seki \cite{MS26}, who went further, showing that $\log_\cA(\alpha)$ cannot take non-zero rational values in $\cA$ for $\alpha\in\mathbb{Q}^\times\setminus\{\pm 1\}$. Furthermore, an earlier result of Silverman \cite{Si88} on non-Wieferich primes shows that $\log_\cA(\alpha)$ cannot be zero, provided the $abc$-conjecture holds.
\begin{theorem}[Matsusaka--Seki, Silverman]\label{MSS}
For any $\alpha\in\mathbb{Q}^\times\setminus\{\pm 1\}$, we have 
\begin{itemize}
\item[(i)]
$\log_\cA(\alpha)\in\cA\setminus\mathbb{Q}^{\times}$, 
\item[(ii)]
$\log_\cA(\alpha)\neq 0\in\cA$ assuming the $abc$-conjecture.
\end{itemize}
\end{theorem}
In \cite{GRM13}, Graves and Ram Murty extended Silverman's method to prove Theorem \ref{MSS}(ii) for integer $\alpha>1$ when primes are restricted to congruence classes 
\[\cP(m)=\{p\in\cP\;|\; p\equiv 1\bmod m\}\]
for fixed integer $m\geq 1$. In this note, we will extend this to rational $\alpha$ and give a similar extension to Theorem \ref{MSS}(i) for these congruence classes. 
To state this, define a corresponding restricted poor man's ad\`{e}le ring
\[\cA(m):=\left.\left(\;\prod_{p\in\cP(m)}\mathbb{Z}/p\mathbb{Z}\right) \middle/ \left(\;\bigoplus_{p\in\cP(m)}\mathbb{Z}/p\mathbb{Z}\right)\right..\]
The set $\cP(m)$ is infinite with density $1/\varphi(m)$ in $\cP$ by Chebotarev's density theorem, where $\varphi(m)$ denotes Euler's totient function. 
Again, the field of rational numbers $\mathbb{Q}$ embeds in $\cA(m)$ naturally, and we consider the corresponding logarithm 
\[\log_{\cA(m)}(\alpha)=(q_p(\alpha)\bmod p)_{p\in\cP(m)}.\]
\begin{theorem}\label{logAm}
For any $\alpha\in\mathbb{Q}^\times\setminus\{\pm 1\}$ and $m\geq 1$, we have
\begin{itemize}
\item[(i)]
$\log_{\cA(m)}(\alpha)\in\cA(m)\setminus\mathbb{Q}^{\times}$, 
\item[(ii)]
$\log_{\cA(m)}(\alpha)\neq 0\in\cA(m)$ assuming the $abc$-conjecture.
\end{itemize}
\end{theorem}
Theorem \ref{MSS} can be rephrased in the following way: for fixed $\alpha\in\mathbb{Q}^\times\setminus\{\pm 1\}$ and $c\in\mathbb{Q}$, there are infinitely many primes $p$ such that 
\[q_p(\alpha)=\frac{\alpha^{p-1}-1}{p}\not\equiv c\bmod p\]
where the case $c=0$ is dependent on the $abc$-conjecture. Theorem \ref{logAm} says there are infinitely many primes satisfying this condition in arithmetic progressions of the form $p\equiv 1\bmod m$. 
The methods in both \cite{MS26} and \cite{Si88} revolve around finding suitable congruences and bounds for values of cyclotomic polynomials $\Phi_\ell(\alpha)$ with sufficiently large primes $\ell$, assuming $\log_\cA(\alpha)$ is rational. We modify this by considering $\Phi_{m\ell}(\alpha)$, together with a suitable logarithmic derivative correction.
\smallskip

The machinery used to prove Theorem \ref{logAm} can be adapted to give further results about $\log_\cA(\alpha)$ in $\cA$. 
Rather than excluding a single non-zero rational value $q_p(\alpha)=c$, a signed version of the cyclotomic argument allows us to exclude the possibility that $q_p(\alpha)=\pm c$ simultaneously for $c\in\mathbb{Q}^\times$. 
In particular, we obtain a second strengthening of Theorem \ref{MSS}.
\begin{theorem}\label{logsquare}
For any $\alpha\in\mathbb{Q}^\times\setminus\{\pm 1\}$, we have 
\begin{itemize}
\item[(i)]
$\log_\cA(\alpha)^2\in\cA\setminus\mathbb{Q}^{\times}$, 
\item[(ii)]
$\log_\cA(\alpha)^2\neq 0\in\cA$ assuming the $abc$-conjecture.
\end{itemize}
\end{theorem}

The restricted congruence class result also has an application concerning $\log_\cA(\alpha)$ in Rosen's theory of finite algebraic numbers. Following a suggestion by Luca and Zudilin in \cite{LZ25}, we define the recurrence degree $\deg_\cA({\bf t})$ of ${\bf t}\in\cA$ to be the least order of a rational constant-coefficient recurrence whose $p$-th terms represent ${\bf t}$. We obtain the following consequence of Theorem \ref{logAm}.
\begin{theorem}\label{degtwo}
Let $\alpha\in\mathbb{Q}^\times\setminus\{\pm 1\}$. If $\deg_\cA(\log_\cA(\alpha))\leq 2$, 
then there exists $r\in\mathbb{Q}$ and a square-free integer $D\neq 1$ such that
\[\log_\cA(\alpha)=r\left(1-\leg{D}{p}\right)_p.\]
Furthermore, if the $abc$-conjecture holds, then $\deg_\cA(\log_\cA(\alpha))>2$. 
\end{theorem}
Section 2 contains necessary cyclotomic constructions and we give the proofs of Theorem \ref{logAm} and Theorem \ref{logsquare} in Section 3. In Section 4, we introduce the recurrence degree, classify elements of degree at most 2, and use this to prove Theorem \ref{degtwo}.

\section{Cyclotomic preliminaries}
For integer $n\geq 1$, let $\zeta_n$ be a primitive $n$th root of unity and $\Phi_n(T)$ be the $n$th cyclotomic polynomial with degree $\varphi(n)$
\begin{equation*}
\Phi_n(T)=\prod_{\substack{1\leq j<n \\ \gcd(j,n)=1}} \left(T-\zeta_n^j\right)\in\mathbb{Z}[T].
\end{equation*}
We will also make use of the homogenised form
\begin{equation}\label{homPhi}
\Phi_{n}(X,Y)=Y^{\varphi(n)}\Phi_n(X/Y)=
\prod_{\substack{1\leq j<n \\ \gcd(j,n)=1}} \left(X-\zeta_n^jY\right)\in\mathbb{Z}[X,Y]
\end{equation}
as well as its partial derivatives $\Phi_{n,X}(X,Y)$ and $\Phi_{n,Y}(X,Y)$. By Euler's homogeneous form identity, these satisfy
\begin{equation}\label{euler}
X\Phi_{n,X}(X,Y)+Y\Phi_{n,Y}(X,Y)=\varphi(n)\Phi_n(X,Y).
\end{equation}
Given $n\geq 1$ and a prime $p\nmid n$, the well-known identity $\Phi_n(T^p)=\Phi_n(T)\Phi_{np}(T)$ has homogenised version
\begin{equation}\label{Phinp}
\Phi_n(X^p,Y^p)=\Phi_n(X,Y)\Phi_{np}(X,Y). \\[10pt]
\end{equation}
To investigate $q_p(\alpha)$ for $\alpha=u/v\in\mathbb{Q}^\times\setminus\{\pm 1\}$, since 
\[q_p(\alpha)\equiv q_p(-\alpha)\equiv -q_p(\alpha^{-1})\bmod p\]
for $p>2$, we need only consider $u>v\geq 1$ with $\gcd(u,v)=1$, and will assume this from here onwards. 
We first collect some standard facts regarding prime factors of $\Phi_n(u,v)$. 
\begin{lemma}\label{ordlem}
Given prime $p$ dividing $\Phi_n(u,v)$, we have the following.
\begin{itemize}
\item[(i)]
$p\nmid uv$,
\item[(ii)]
if $p\nmid n$, then $\ord_p(\alpha)=n$ and $p\equiv 1\mod n$,
\item[(iii)]
if $p\mid n$, then $\ord_p(\alpha)=n/p^k$ for some $k\geq 1$ and $p$ is the largest prime factor of $n$.
\end{itemize}
\end{lemma}
\begin{proof}
Since $\Phi_n(u,v)\mid \left(u^n-v^n\right)$, we have $u^n\equiv v^n\bmod p$, so if $p$ divided one of $u$ or $v$, it would divide both. This contradicts $\gcd(u,v)=1$ and part (i) follows.
\medskip

The reduction $\alpha\bmod p$ is a well-defined element in $\mathbb F_p^\times$ and $p$ divides $\Phi_n(u,v)=v^{\varphi(n)}\Phi_{n}(\alpha)$,  thus $\Phi_{n}(\alpha)\equiv 0\bmod p$. 
Writing $n=n_0p^k$ where $p\nmid n_0$, the following identity in $\mathbb{F}_p[T]$
\[\Phi_n(T)\equiv\Phi_{n_0}(T)^{\varphi(p^k)}\bmod p\]
can be proved easily using \eqref{Phinp} or via M\"{o}bius inversion. Thus $\Phi_{n_0}(\alpha)\equiv 0\bmod p$ and we have $\alpha^{n_0}\equiv 1\bmod p$. 
Suppose that $\alpha^d-1\equiv 0\bmod p$ for some $d\mid n_0$. Then $\alpha$ is a root modulo $p$ of 
\[T^{d}-1=\prod_{e\mid d}\Phi_e(T)\]
and hence $\Phi_e(\alpha)\equiv 0\bmod p$ for some $e\mid d$ and $d\mid n_0$. 
However, the polynomial $T^{n_0}-1$ is separable over $\mathbb{F}_p$ and the factors in the decomposition
\[T^{n_0}-1=\prod_{d\mid n_0}\Phi_d(T)\]
are pairwise coprime over $\mathbb{F}_p$. We deduce that $e=d=n_0$ and so $\ord_p(\alpha)=n_0=n/p^k$. 
\medskip

Clearly this order $n_0$ divides $\left|\mathbb{F}_p^\times\right|=p-1$ and when $k=0$, we obtain $p\equiv 1\bmod n$. 
When $k\geq 1$, any prime $q\neq p$ dividing $n$ divides $n_0$ and so $q\mid p-1$. Hence $p$ is the largest prime factor of $n$.
\end{proof}

We next relate Fermat quotients with values of cyclotomic polynomials. 
\begin{lemma}\label{conglem}
Suppose that prime $p\mid\Phi_n(u,v)$ and $p\nmid n$. Then
\[\dfrac{\Phi_{n}(u,v)}{p}\equiv q_p(\alpha)v\Phi_{n,Y}(u,v)\bmod p\] 
and moreover, $\Phi_{n,Y}(u,v)\not\equiv 0\bmod p$.
\end{lemma}
\begin{proof} 
Using \eqref{euler} and the assumption $p\mid\Phi_n(u,v)$, we have
\begin{equation}\label{euler2}
u\Phi_{n,X}(u,v)\equiv -v\Phi_{n,Y}(u,v)\bmod p.
\end{equation}
Now, $u^p-u\equiv v^p-v\equiv 0\bmod p$, so by taking a Taylor expansion we find
\begin{equation}\begin{aligned}[b]
\Phi_n(u^p,v^p) 
&\equiv \Phi_n(u,v)+(u^p-u)\Phi_{n,X}(u,v)+(v^p-v)\Phi_{n,Y}(u,v)\bmod p^2  \\
&\equiv \Phi_n(u,v)+p\left[q_p(u)u\Phi_{n,X}(u,v)+q_p(v)v\Phi_{n,Y}(u,v)\right]\bmod p^2 \\
&\equiv\Phi_n(u,v)-p\left(q_p(u)-q_p(v)\right)v\Phi_{n,Y}(u,v) \bmod p^2 \\
&\equiv\Phi_n(u,v)-pq_p(\alpha)v\Phi_{n,Y}(u,v)\bmod p^2. \label{cong}
\end{aligned}\end{equation}
Note that since $p\nmid n$, we have
\[\Phi_n(X,Y)\Phi_{np}(X,Y)=\Phi_n(X^p,Y^p)\equiv\Phi_n(X,Y)^p\bmod p.\]
and so $\Phi_{np}(X,Y)\equiv\Phi_n(X,Y)^{p-1}\bmod p$. 
In particular, $\Phi_{np}(u,v)\equiv\Phi_{n}(u,v)^{p-1}\equiv 0\bmod p$ and hence 
\[\Phi_n(u^p,v^p)=\Phi_n(u,v)\Phi_{np}(u,v)\equiv 0 \bmod p^2.\]
The congruence follows on dividing \eqref{cong} by $p$.
\medskip

By the previous lemma, $\alpha$ is a simple root of $\Phi_n(T)\bmod p$ so we have $\Phi_n'(\alpha)\not\equiv 0\bmod p$. Hence
\[\Phi_{n,X}(u,v)=v^{\varphi(n)-1}\Phi_n'(\alpha)\not\equiv 0\bmod p\] 
and $\Phi_{n,Y}(u,v)\not\equiv 0\bmod p$ follows from \eqref{euler2}.
\end{proof}
\medskip

Finally, we need a lemma controlling primes $p\mid\Phi_n(u,v)$ as $n$ increases. We first record the trivial estimates
\begin{equation}\label{trivest}
(u-v)^{\varphi(n)}\leq\Phi_n(u,v)\leq (u+v)^{\varphi(n)}
\end{equation}
which follow from the product expansion in \eqref{homPhi}. Note we are assuming $u>v\geq 1$ throughout so $\Phi_n(u,v)$ is positive.
\begin{lemma}\label{limlem}
For $n>1$, let $r_n=\#\{p\text{ prime} : p\mid\Phi_n(u,v), p\nmid n\}$. Then 
\[r_n<\frac{\varphi(n)\log(u+v)}{\log n}
\qquad\text{and}\qquad
\lim_{n\to\infty}\sum_{\substack{p\mid\Phi_n(u,v) \\ p\nmid n}}\frac{1}{p}=0.\]
\end{lemma}
\begin{proof}
When $r_n=0$, the inequality holds trivially so assume $r_n>0$. For each prime $p$ dividing $\Phi_n(u,v)$ with $p\nmid n$, 
we have $p\equiv 1\mod n$ by Lemma \ref{ordlem}, so $p>n$. By \eqref{trivest},
\begin{align*}
n^{r_n}<\prod_{\substack{p\mid\Phi_n(u,v) \\ p\nmid n}}\!\! p\;\ \leq \Phi_n(u,v)
\leq (u+v)^{\varphi(n)}
\end{align*}
and the inequality for $r_n$ follows. Furthermore, $\varphi(n)<n$ so
\[\sum_{\substack{p\mid\Phi_n(u,v) \\ p\nmid n}}\frac{1}{p}
\leq\frac{r_n}{n}
<\frac{\log(u+v)}{\log n}\longrightarrow 0\qquad\text{as $n\to\infty$.}\]
\vspace{-1cm}

\end{proof}

\section{Irrationality of finite logarithms}
\subsection{Non-zero rational values in $\cA(m)$} $\,$\\[3pt]
To prove Theorem \ref{logAm}(i), we adapt the method used in \cite{MS26} to fit the template of many classical irrationality proofs by contradiction: assuming rationality, define a sequence of non-zero integers that tends to zero.
\begin{proof}[Proof of Theorem \ref{logAm}(i)] 
Fix $m\geq 1$ and suppose for contradiction there exist $a, b\in\mathbb{Z}$ with $a\neq 0$, $b\geq 1$, $\gcd(a,b)=1$ and a bound $N>\max\{|a|,b\}$ such that for all primes $p>N$ with $p\equiv 1\bmod m$,  we have
\[q_p(\alpha)\equiv \frac{a}{b}\bmod p.\] 
\vspace{-0.1cm}

\noindent
Let $w=\Phi_m(u,v)$ and consider $W_\ell=\Phi_{m\ell}(u,v)$ for primes $\ell>\max\{N, m, (u+v)^{\varphi(m)}\}$. 
Since $\ell\nmid m$, by \eqref{Phinp} we have  
\[W_\ell=\Phi_{m\ell}(u,v)=\frac{\Phi_m(u^\ell,v^\ell)}{w}.\]
We also note $\ell>w$ since $(u+v)^{\varphi(m)}\geq w$ by \eqref{trivest}.
\vspace{0.2cm}

Given a prime $p\mid W_\ell$, we automatically have $p\nmid m\ell$. Indeed, if $p\mid m\ell$, then by Lemma \ref{ordlem}(iii) the largest prime factor of $m\ell$ would be $\ell=p$. However, 
\[wW_\ell=\Phi_m(u^\ell,v^\ell)\equiv\Phi_m(u,v)^\ell\equiv w^\ell\mod\ell\]
and so $W_\ell\equiv w^{\ell-1}\equiv 1\mod\ell$, contradicting $\ell=p\mid W_\ell$. 
Consequently, $p\equiv 1\mod m\ell$ by Lemma \ref{ordlem}(ii) so $p\equiv 1\mod m$ and $p>\ell>N$.
Hence $q_p(\alpha)\equiv ab^{-1}\not\equiv 0\bmod p$ and Lemma \ref{conglem} therefore gives
\begin{equation*}\label{Wlcong}
\frac{W_\ell}{p}\equiv q_p(\alpha)v\Phi_{m\ell,Y}(u,v)\not\equiv 0\bmod p.
\end{equation*}
This means $p^2\nmid W_\ell$ for every prime $p\mid W_\ell$ and so $W_\ell$ is square-free.
\bigskip

Differentiate $\Phi_m(X^\ell,Y^\ell)=\Phi_m(X,Y)\Phi_{m\ell}(X,Y)$ with respect to $Y$ to get
\[\ell Y^{\ell-1}\Phi_{m,Y}(X^\ell,Y^\ell)=\Phi_m(X,Y)\Phi_{m\ell,Y}(X,Y)+\Phi_{m\ell}(X,Y)\Phi_{m,Y}(X,Y).\]
Using this, define the following integer
\[R_\ell=\ell v^\ell\Phi_{m,Y}(u^\ell,v^\ell)=wv\Phi_{m\ell,Y}(u,v)+W_\ell v\Phi_{m,Y}(u,v)\]
and for each $p\mid W_\ell$, we now have a congruence
\begin{equation*}\label{wWlcong}
w\frac{W_\ell}{p}\equiv q_p(\alpha)R_\ell\bmod p.
\end{equation*}
Given that $q_p(\alpha)\equiv ab^{-1}\bmod p$, for each prime $\ell$ as above, set
\begin{equation}\label{Tell}
T_\ell=bwS_\ell-aR_\ell\in\mathbb{Z}\qquad\text{where}\qquad S_\ell=\sum_{p\mid W_\ell}\frac{W_\ell}{p}\in\mathbb{Z}.
\end{equation}
Notice that for each specific prime $p\mid W_\ell$, the summands $W_\ell/p'$ for $p'\neq p$ are divisible by $p$ so
$S_\ell\equiv W_\ell/p\mod p$. In particular, 
\[T_\ell\equiv bw\frac{W_\ell}{p}-aR_\ell\equiv bq_p(\alpha)R_\ell-aR_\ell\equiv 0\bmod p.\]
Thus primes dividing $W_\ell$ also divide $T_\ell$, and since $W_\ell$ is square-free, we obtain $T_\ell/W_\ell\in\mathbb{Z}$.
\bigskip

We next show the sequence $T_\ell/W_\ell$ tends to zero. By Lemma \ref{limlem}, the sum $\sum_{p\mid W_\ell}1/p\to 0$ as $\ell\to\infty$. Also,
\[\frac{R_\ell}{W_\ell}=\ell wv^\ell\frac{\Phi_{m,Y}(u^\ell,v^\ell)}{\Phi_m(u^\ell,v^\ell)}
=\ell w\alpha^{-\ell}\frac{\Phi_{m,Y}(1,\alpha^{-\ell})}{\Phi_m(1,\alpha^{-\ell})}
\]
where we have used that $\Phi_m(X,Y)$ is homogeneous of degree $\varphi(m)$ and $\Phi_{m,Y}(X,Y)$ is homogeneous of degree $\varphi(m)-1$. The logarithmic derivative factor is bounded. Indeed, taking the logarithmic $Y$-derivative of \eqref{homPhi} with $n=m$ and setting $X=1$ and $Y=\alpha^{-\ell}$, we have 
\[\left|\frac{\Phi_{m,Y}(1,\alpha^{-\ell})}{\Phi_m(1,\alpha^{-\ell})}\right|\leq
\sum_{\substack{1\leq j<m \\ \gcd(j,m)=1}}\left|\frac{-\zeta_m^j}{1-\zeta_m^j\alpha^{-\ell}}\right|
\leq\frac{\varphi(m)}{1-\alpha^{-\ell}}.\]
Hence 
\[\left|\frac{R_\ell}{W_\ell}\right|\leq\ell w\varphi(m)\frac{\alpha^{-\ell}}{1-\alpha^{-\ell}}=\frac{lw\varphi(m)}{\alpha^\ell-1}\]
and this tends to zero since $\alpha=u/v>1$. Overall, 
\[\left|\frac{T_\ell}{W_\ell}\right|\leq bw\sum_{p\mid W_\ell}\frac{1}{p}+|a|\left|\frac{R_\ell}{W_\ell}\right|\longrightarrow 0
\]
as primes $\ell\to\infty$.
\medskip

Finally, we show $T_\ell\not\equiv 0\mod \ell$. 
As before, we have $W_\ell\equiv 1\bmod \ell$ and each prime divisor $p$ of $W_\ell$ satisfies $p\equiv 1\mod \ell$ and hence $W_\ell/p\equiv 1\mod\ell$. By definition, $\ell\mid R_\ell$ so
\[T_\ell\equiv bwr_{m\ell}\bmod \ell\]
using the notation from Lemma \ref{limlem}. We have chosen $\ell$ so that $1\leq b<\ell$ and $1\leq w<\ell$ so $bw\not\equiv 0\bmod\ell$. We also chose $\ell>(u+v)^{\varphi(m)}$ and hence by Lemma \ref{limlem}, 
\[r_{m\ell}<\frac{(\ell-1)\varphi(m)\log(u+v)}{\log(m\ell)}<\ell-1.\]
Furthermore, by \eqref{trivest}, 
\[W_\ell=\frac{\Phi_m(u^\ell,v^\ell)}{\Phi_m(u,v)}\geq\left(\frac{u^\ell-v^\ell}{u+v}\right)^{\varphi(m)}>1
\]
since $u^\ell-v^\ell>u+v$ for every $\ell\geq 3$. Thus $W_\ell$ has a prime divisor and $1\leq r_{m\ell}<\ell$. 
We conclude that $T_\ell\not\equiv 0\bmod \ell$. In particular, $T_\ell/W_\ell$ is a non-zero integer and we obtain the required contradiction.
\end{proof}

\subsection{Zero values in $\cA(m)$} $\,$\\[3pt]
Graves and Ram Murty (Theorem 3.1 in \cite{GRM13}) demonstrate that there are infinitely many primes $p\equiv 1\bmod m$ such that $q_p(\alpha)\not\equiv 0\bmod p$ for integer $\alpha>1$, assuming the $abc$-conjecture. In fact, they show that the number of such primes not exceeding $x$ is $\gg \frac{\log x}{\log\log x}$ and this bound was later improved to $\gg\log x$ by Ding \cite{D19}. For completeness, we will briefly show there are infinitely many such primes for arbitrary rational $\alpha\in\mathbb{Q}^\times\setminus\{\pm 1\}$. 

The input from the $abc$-conjecture appears via the following result used in Silverman's proof of Lemma 7 in \cite{Si88}.
\begin{lemma}\label{Silabc}
Let $\alpha=u/v$ where $u>v\geq 1$ with $\gcd(u,v)=1$. For each $n\geq 1$, write $u^n-v^n=A_nB_n$ where $B_n$ is the powerful part of $u^n-v^n$, that is
\[A_n=\prod_{p||u^n-v^n} p\qquad\text{and}\qquad 
  B_n=\prod_{\substack{p^e||u^n-v^n \\ e\geq 2}}p^e.\]
Assuming the $abc$-conjecture, for every $\epsilon>0$ we have
$B_n\ll_{\alpha,\epsilon} u^{\epsilon n}$.
\end{lemma}
\begin{proof}
Since
\[\rad(u^nv^n(u^n-v^n))\leq \rad(uv)A_nB_n^{1/2}=\rad(uv)\frac{u^n-v^n}{B_n^{1/2}}\leq\rad(uv)\frac{u^n}{B_n^{1/2}},\]
applying the $abc$-conjecture to
\[v^n+(u^n-v^n)=u^n\]
gives
\[u^n\ll_{\alpha,\delta}\left(\frac{u^n}{B_n^{1/2}}\right)^{1+\delta}\]
for every $\delta>0$. As a result, we have 
\[B_n^{(1+\delta)/2}\ll_{\alpha,\delta}u^{\delta n}\quad\implies\quad
B_n\ll_{\alpha,\delta}u^{\frac{2\delta}{1+\delta}n}.\]
The lemma follows by taking sufficiently small $\delta$. 
\end{proof}

\begin{proof}[Proof of Theorem \ref{logAm}(ii)]
Suppose the $abc$-conjecture holds. 
As before, we can assume that $\alpha=u/v$ with $u>v\geq 1$ and $\gcd(u,v)=1$. 
Suppose for contradiction that there exists $N$ such that $q_p(\alpha)\equiv 0\bmod p$ for each prime $p>N$ with $p\equiv 1\mod m$. 
\medskip

We consider primes $\ell>\max\{N,m,(u+v)^{\varphi(m)}\}$ and set 
\[W_\ell=\Phi_{m\ell}(u,v)=\frac{\Phi_m(u^\ell,v^\ell)}{\Phi_m(u,v)}.\] 
As in the previous part, every prime $p\mid W_\ell$ satisfies $p\nmid m\ell$ and $p\equiv 1\bmod m\ell$. In particular, 
\[p>\ell>N\qquad\text{and}\qquad p\equiv 1\mod m\]
so our assumption gives $q_p(\alpha)\equiv 0\bmod p$. Applying Lemma \ref{conglem} with $n=m\ell$, we obtain
\[\frac{W_\ell}{p}\equiv q_p(\alpha)v\Phi_{m\ell,Y}(u,v)\equiv 0\bmod p.\]
As a result, $p^2\mid W_\ell$ for every $p\mid W_\ell$ and $W_\ell$ is square-full (equivalently, powerful). 
Also, since $W_\ell\mid \left(u^{m\ell}-v^{m\ell}\right)$, we have 
\[W_\ell\mid B_{m\ell}\]
where $B_{m\ell}$ is the powerful part of $u^{m\ell}-v^{m\ell}$. Applying Lemma \ref{Silabc} with $\epsilon=\varphi(m)/2m$, we obtain
\[W_\ell\leq B_{m\ell}\ll_{\alpha,m}u^{\ell\varphi(m)/2}.\]
On the other hand, the cyclotomic estimates \eqref{trivest} give
\[W_\ell=\frac{\Phi_m(u^\ell,v^\ell)}{\Phi_m(u,v)}
\geq \frac{(u^\ell-v^\ell)^{\varphi(m)}}{(u+v)^{\varphi(m)}}
\geq \left(u^\ell\frac{1-v/u}{u+v}\right)^{\varphi(m)}
\gg_{\alpha,m} u^{\ell\varphi(m)}.\]
These upper and lower bounds for $W_\ell$ are incompatible for sufficiently large $\ell$ so we have a contradiction.
\end{proof}
\begin{remark}
The restriction $p\equiv 1\bmod m$ is intrinsic to the cyclotomic method presented here. 
Indeed, outside exceptional primes, a divisor $p\mid\Phi_{m\ell}(u,v)$ naturally lies in the class $p\equiv 1 \bmod m$. No analogous choice of these cyclotomic values appears to force a prescribed reduced class $p\equiv b\bmod m$ for $b\not\equiv 1\bmod m$ and the $abc$-conditional arguments of \cite{GRM13} and \cite{D19} have the same limitation. This reflects the familiar fact that the infinitude of primes $p\equiv 1\bmod m$ has an elementary cyclotomic proof, whereas Dirichlet's theorem on primes in arithmetic progressions lies much deeper.
\end{remark}

\subsection{Squares of finite logarithms in $\cA$} $\,$\\[3pt]
The condition $\log_\cA(\alpha)^2=c\in\mathbb{Q}^\times$ is equivalent to saying that for sufficiently large primes $p$, we have
\[q_p(\alpha)^2\equiv c\bmod p.\]
Since a non-square rational number is a quadratic non-residue modulo infinitely many primes, we must have $c=(a/b)^2$ for integers $a\neq 0$, $b\geq 1$, $\gcd(a,b)=1$. Then for every sufficiently large prime $p$, there is $\varepsilon_p\in\{\pm 1\}$ such that
\[q_p(\alpha)\equiv \varepsilon_p\frac{a}{b}\bmod p.\]
As before, for $\alpha=u/v\in\mathbb{Q}^\times\setminus\{\pm 1\}$, we need only consider $u>v\geq 1$ with $\gcd(u,v)=1$.
\smallskip

To obtain a contradiction, we will follow the same method as for Theorem \ref{logAm}. The essential differences are that $m$ will be chosen to be a suitable fixed auxiliary prime, the integer \eqref{Tell} will be replaced by a signed version involving $\varepsilon_p$, and $\ell$ will range over suitably large primes with $\ell\equiv 1\bmod m(m-1)$.

\begin{proof}[Proof of Theorem \ref{logsquare}] 
Suppose for contradiction there exists a bound $N>\max\{|a|,b\}$ such that for all primes $p>N$, we have
\[q_p(\alpha)\equiv\varepsilon_p\frac{a}{b}\bmod p \qquad\text{where $\varepsilon_p\in\{\pm 1\}$}.\] 
Fix an odd prime $m$ such that $m\nmid av(u-v)$ and set $w=\Phi_m(u,v)$.
For each prime $\ell>\max\{N,m,(u+v)^{m-1}\}$ with $\ell\equiv 1\bmod m(m-1)$, we set 
\[W_\ell=\Phi_{m\ell}(u,v)\qquad\text{and}\qquad R_\ell=\ell v^\ell\Phi_{m,Y}(u^\ell,v^\ell).\]
As before, any prime $p\mid W_\ell$ satisfies $p\equiv 1\mod m\ell$ and
\[w\frac{W_\ell}{p}\equiv q_p(\alpha)R_\ell\bmod p.\]
Furthermore, the right-hand side is non-zero modulo $p$ and $W_\ell$ is square-free. 
We define a signed version of \eqref{Tell} by 
\begin{equation*}
T_\ell^\varepsilon=bwS_\ell^\varepsilon-aR_\ell\in\mathbb{Z}\qquad\text{where}\qquad S_\ell^\varepsilon=\sum_{p\mid W_\ell}\varepsilon_p\frac{W_\ell}{p}\in\mathbb{Z}.
\end{equation*}
As in the proof of Theorem \ref{logAm}(i), we can conclude that $T_\ell^\varepsilon/W_\ell\in\mathbb{Z}$. 
Indeed, for $p\mid W_\ell$, we have $S_\ell^\varepsilon\equiv \varepsilon_pW_\ell/p\bmod p$ and so 
\[T_\ell^\varepsilon\equiv b\varepsilon_pq_p(\alpha)R_\ell-aR_\ell\equiv 0\bmod p.\]
Similarly, we can conclude that $T_\ell^\varepsilon/W_\ell\to 0$ as $\ell\to\infty$. 
Indeed, this follows since $\left|R_\ell/W_\ell\right|\to 0$ as before, and Lemma \ref{limlem} gives
\[\left|\frac{S_\ell^\varepsilon}{W_\ell}\right|\leq \sum_{p\mid W_\ell}\frac{1}{p}\longrightarrow 0.\]
The only remaining step is to show that $T_\ell^\varepsilon/W_\ell$ is non-zero. 
Set 
\[E_\ell=\sum_{p\mid W_\ell}\varepsilon_p.\]
As before, we have $S_\ell^\varepsilon\equiv E_\ell\bmod \ell$, $R_\ell\equiv 0\bmod \ell$ and hence 
\[T_\ell^\varepsilon\equiv bwE_\ell\bmod \ell.\]
We have chosen $\ell$ so that $1\leq b<\ell$ and $1\leq w<\ell$ so $bw\not\equiv 0\bmod\ell$. 
We also chose $\ell>(u+v)^{m-1}$ and hence by Lemma \ref{limlem},
\[\left|E_\ell\right|\leq r_{m\ell}<\frac{(\ell-1)(m-1)\log(u+v)}{\log(m\ell)}<\ell-1.\]
Thus, if $E_\ell\neq 0$, we deduce that $T_\ell^\varepsilon\not\equiv 0\bmod \ell$.
\medskip

On the other hand, suppose $E_\ell=0$. 
Since every prime $p\mid W_\ell$ satisfies $p\equiv 1 \bmod m$ and $W_\ell$ is square-free, we have $W_\ell/p\equiv 1\bmod m$. Hence $S_\ell^\varepsilon\equiv E_\ell\equiv 0\bmod m$ and
\[T_\ell^\varepsilon=bwS_\ell^\varepsilon-aR_\ell\equiv -aR_\ell\bmod m.\]
Now $\Phi_m(X,Y)\equiv (X-Y)^{m-1}\bmod m$ and so $\Phi_{m,Y}(X,Y)\equiv (X-Y)^{m-2}\bmod m$. 
Furthermore, we are letting $\ell$ range over the infinitely many primes of the form
\[\ell\equiv 1\bmod m(m-1).\] 
Thus $\ell\equiv 1\bmod m$ and $\ell\equiv 1\bmod m-1$ so $u^\ell\equiv u\bmod m$ and $v^\ell\equiv v\bmod m$. 
In particular, we now have
\[R_\ell=\ell v^\ell\Phi_{m,Y}(u^\ell,v^\ell)\equiv v(u-v)^{m-2}\not\equiv 0\bmod m.\]
Thus, if $E_\ell=0$, we deduce that $T_\ell^\varepsilon\not\equiv 0\bmod m$.
\medskip

In either case, we find $T_\ell^\varepsilon/W_\ell$ is a non-zero integer which tends to zero as $\ell\to\infty$, giving the required contradiction. 
This completes the proof of Theorem \ref{logsquare}(i). 
Finally, part (ii) follows immediately from Theorem \ref{MSS}(ii), since $\cA$ is a reduced ring: for any ${\bf t}\in\cA$, if ${\bf t}^2=0$ then ${\bf t}=0$.
\end{proof}

\begin{remark}
There is an alternative finite logarithm that can be considered in $\cA$, as well as finite analogues of higher polylogarithms. 
For prime $p$ and integer $d\geq 1$, the {\it finite $d$-logarithm} is
\[\LL_{d}(X)=\sum_{k=1}^{p-1}\frac{X^k}{k^d}.\]
For $d=1,2$, these are related to Fermat quotients modulo $p$ and in particular (see e.g. \cite{Gr04}), for $p>3$ we have simple identities
\[\LL_{1}(2)\equiv -2q_p(2)\bmod p\qquad\text{and}\qquad\LL_{2}(2)\equiv -q_p(2)^2\bmod p.\]
In view of Theorems \ref{MSS} and \ref{logsquare}, we deduce $\left(\LL_{1}(2)\bmod p\right)_p$ and $\left(\LL_{2}(2)\bmod p\right)_p$ are not in $\mathbb{Q}^\times\subset\cA$ unconditionally and, assuming the $abc$-conjecture, are not rational elements of $\cA$.
\end{remark}

\section{Quadratic finite algebraic numbers}
The restriction to primes in arithmetic progressions also distinguishes finite logarithms from certain natural algebraic elements of $\cA$ in the context of Rosen's theory of finite algebraic numbers \cite{Ro20}. Let $\chi$ be a non-trivial quadratic Dirichlet character with conductor $N$. Then 
for any $\alpha\in\mathbb{Q}^\times\setminus\{\pm 1\}$ and $c\in\mathbb{Q}^\times$, we have
\[\log_\cA(\alpha)\neq (c\chi(p))_p.\]
Indeed, if we had equality, then restricting to sufficiently large primes $p\equiv 1\bmod N$ gives $\log_{\cA(N)}(\alpha)=c$, contradicting Theorem \ref{logAm}(i). In this section, we will use the same idea to show that, provided the $abc$-conjecture holds, $\log_\cA(\alpha)$ cannot be a quadratic irrational element of $\cA$. 
To this end, we need to first say what we mean by a quadratic element in $\cA$. 
Following the set-up in \cite{RTTY24}, given a monic polynomial
\[f(x)=x^d+c_1x^{d-1}+\cdots +c_d\in\mathbb{Q}[x]\] 
one defines the $\mathbb{Q}$-vector space
\[\Rec(f;\mathbb{Q})=\left\{(a_n)_n\in\prod_{n\geq 0}\mathbb{Q}\; \middle|\; a_n+c_1a_{n-1}+\cdots+c_da_{n-d}=0
\text{ for all $n\geq d$}\right\}.\]
The set of primes occurring in the denominators of initial values $a_0,\dots,a_{d-1}$ and in the denominators of the coefficients of $f$ is finite. Outside of this set, every $a_p$ is $p$-integral so there is a well-defined $\mathbb{Q}$-linear map 
\begin{align*}
r_f:\Rec(f;\mathbb{Q})&\longrightarrow\cA \\
(a_n)_n&\longmapsto (a_p\bmod p)_p
\end{align*}
The union of the images of $r_f$ as $f$ ranges over all monic polynomials in $\mathbb{Q}[x]$ is precisely the set of finite algebraic numbers in $\cA$ defined in \cite{Ro20}. 
For ${\bf t}\in\cA$, we define its recurrence degree to be the minimal order of all possible recurrences representing the element ${\bf t}$ 
\[\deg_{\cA}({\bf t})=\min\left\{\deg f \;\middle|\; f(x)\in\mathbb{Q}[x]\text{ monic, non-constant and ${\bf t}\in\im(r_f)$}  \right\}\]
with the convention that $\deg_{\cA}({\bf t})=\infty$ when no such $f$ exists. 
We have the following simple criterion for rational and quadratic elements.
\begin{proposition}\label{quadinA}
An element ${\bf t}\in\cA$ satisfies $\deg_\cA({\bf t})\leq 2$ if and only if there exist $r,s\in\mathbb{Q}$ and a square-free integer $D\neq 1$ such that
\[{\bf t}=\left(r+s\leg{D}{p}\right)_p.\]
Moreover, $\deg_\cA({\bf t})=2$ if and only if $s\neq 0$.
\end{proposition}
\begin{proof}
Degree 1 elements are precisely the rational elements of $\cA$. Indeed, for $f(x)=x-c$ the corresponding sequence is $a_n=a_0c^n$, and for every sufficiently large prime $p$,
\[a_p=a_0c^p\equiv a_0c\bmod p.\]
Thus $r_f((a_n)_n)$ is rational. Conversely, rational elements in $\cA$ are represented by constant sequences which satisfy the order $1$ recurrence $a_n-a_{n-1}=0$. 
Now suppose that ${\bf t}$ is represented by a sequence $(a_n)_n\in\Rec(f;\mathbb{Q})$ where $f$ is quadratic. 
\smallskip

If $f$ is reducible over $\mathbb{Q}$, then for $n\geq 2$ we have
\[a_n=u\lambda^n+v\mu^n\qquad\text{or}\qquad a_n=(u+vn)\lambda^n\]
for some $u,v,\lambda,\mu\in\mathbb{Q}$. Fermat's little theorem shows $a_p\bmod p$ is constant for sufficiently large primes $p$, and hence ${\bf t}\in\mathbb{Q}\subset\cA$.
\smallskip

If $f$ is irreducible, let $L=\mathbb{Q}(\sqrt{D})$ be its splitting field where $D\neq 1$ is square-free, and let its roots be $\lambda,\overline{\lambda}$. 
Then $a_n=u\lambda^n+\overline{u}\overline{\lambda}^n$ for some $u\in L$ and we define
\[c_+=u\lambda+\overline{u}{\overline{\lambda}}\qquad\text{and}\qquad
c_-=u\overline{\lambda}+\overline{u}\lambda.\]
These are clearly rational since they are fixed by conjugation. 
For every sufficiently large prime $p$, the Frobenius map of $L$ at $p$ either fixes or interchanges $\lambda,\overline{\lambda}$, according as $\leg{D}{p}=+1$ or $-1$. 
Thus $a_p\equiv c_+\bmod p$ in the first case and $a_p\equiv c_-\bmod p$ in the second, and we obtain 
\[{\bf t}=\left(r+s\leg{D}{p}\right)_p\qquad\text{where}\; 
r=\frac{c_++c_-}{2}\;\;\text{and}\;\; s=\frac{c_+-c_-}{2}.\]

Conversely, suppose ${\bf t}=\left(r+s\left(\frac{D}{p}\right)\right)_p$ where $D\neq 1$ is square-free. Set 
\[L=\mathbb{Q}(\sqrt{D}),\qquad\lambda=1+\sqrt{D},\qquad u=\frac{1}{2}\left(r+\frac{s}{D}\sqrt{D}\right)\] 
and define the sequence $a_n=u\lambda^n+\overline{u}\overline{\lambda}^n$.
Then $a_n\in\mathbb{Q}$ and $(a_n)_n\in\Rec(f;\mathbb{Q})$ where
\[f(x)=(x-\lambda)(x-\overline{\lambda})=x^2-2x+1-D.\]
Moreover, the Frobenius map of $L$ at $p$ acts on $\sqrt{D}$ via
$\sqrt{D}\mapsto\leg{D}{p}\sqrt{D}$ 
and we obtain 
\[a_p\equiv r+s\leg{D}{p}\bmod p\]
for every sufficiently large prime $p$. Thus $\deg_\cA({\bf t})\leq 2$ as required. 
We remark that this is essentially the construction appearing in Example 3.2 of \cite{RTTY24} where there are two quadratic Frobenius idempotents
\[\frac{1}{2}\left(1+\leg{D}{p}\right)_p\qquad\text{and}\qquad
\frac{1}{2}\left(1-\leg{D}{p}\right)_p.\]

Finally, if $s=0$, then ${\bf t}=r\in\mathbb{Q}$ and $\deg_\cA({\bf t})=1$. 
If ${\bf t}=q\in\mathbb{Q}$, then comparison on infinitely many split and inert primes give $q=r+s$ and $q=r-s$ respectively, forcing $s=0$. 
Hence $s\neq 0$ implies $\deg_\cA({\bf t})=2$.
\end{proof}
\begin{proof}[Proof of Theorem \ref{degtwo}]
By Proposition \ref{quadinA}, we have $r,s\in\mathbb{Q}$ and a square-free $D\neq 1$ such that
\[\log_\cA(\alpha)=\left(r+s\leg{D}{p}\right)_p.\]
Let $N$ be the conductor of the quadratic character associated with $L=\mathbb{Q}(\sqrt{D})$. 
For every sufficiently large prime $p\equiv 1\bmod N$, we have $\leg{D}{p}=1$. 
In particular, restricting to these primes, we have $\log_{\cA(N)}(\alpha)=r+s\in\mathbb{Q}$. 
By Theorem \ref{logAm}(i), this cannot be a non-zero rational number so $r+s=0$ and 
\[\log_\cA(\alpha)=r\left(1-\leg{D}{p}\right)_p\]
as asserted. 
Now suppose that the $abc$-conjecture holds. Since $\leg{D}{p}=1$ for every sufficiently large prime $p\equiv 1\bmod N$, the preceding formula gives $\log_{\cA(N)}(\alpha)=0$, contradicting Theorem \ref{logAm}(ii). We conclude that $\deg_\cA(\log_\cA(\alpha))>2$.
\end{proof}
\begin{remark}
Assuming the $abc$-conjecture, the same argument shows that $\log_\cA(\alpha)\not\in\im(r_f)$ whenever the splitting field $L/\mathbb{Q}$ of $f(x)$ is abelian. 
Briefly, suppose that $(a_n)_n\in\Rec(f;\mathbb{Q})$ represents $\log_\cA(\alpha)$. 
After omitting any eventually vanishing contributions from zero roots, we may write
\[a_n=\sum_i P_i(n)\lambda_i^n\]
where $\lambda_i$ are the distinct roots of $f(x)$ in $L$ and $P_i(x)\in L[x]$. 
By the Kronecker-Weber theorem, $L\subseteq\mathbb{Q}(\zeta_m)$ for some $m$ and, for sufficiently large prime $p\equiv 1\bmod m$, the Frobenius map acts trivially on $L$. Hence
\[a_p\equiv c=\sum_i P_i(0)\lambda_i\bmod p\]
where $c$ is Galois invariant, so $c\in\mathbb{Q}$. It follows that $\log_{\cA(m)}(\alpha)=c$ and Theorem \ref{logAm}(i) implies $c=0$, contradicting Theorem \ref{logAm}(ii). 

Given $\alpha\in\mathbb{Q}^\times\setminus\{\pm 1\}$, it is natural to expect that $\log_\cA(\alpha)\not\in\im(r_f)$ for any characteristic polynomial $f(x)$ and so $\deg_\cA(\log_\cA(\alpha))=\infty$. However, the non-abelian case where splitting is not characterised by congruence conditions is not amenable to a straightforward cyclotomic approach.
\end{remark}
\begin{remark}
Even further, one would expect finite logarithms to be ``naively'' transcendental, i.e. not in the integral closure $\cC_\cA$ of $\mathbb{Q}$ in $\cA$. See \cite{MS26} for multiple results on naive transcendence in $\cA$, as well as the comparison between $\cC_\cA$ and Rosen's finite algebraic numbers. 
We can at least get close to showing $\log_\cA(\alpha)$ is not a root of a quadratic polynomial. 
Firstly, an irreducible quadratic $f(x)\in\mathbb{Q}[x]$ cannot vanish at $\log_\cA(\alpha)$, since its discriminant is a quadratic non-residue modulo infinitely many primes. 
Secondly, a minor modification of the argument in Theorem \ref{logsquare} shows that $\log_\cA(\alpha)$ is not a root of $f(x)=(x-r)(x-s)$ where $r,s\in\mathbb{Q}^\times$. Assuming the $abc$-conjecture, we also know $\log_\cA(\alpha)$ is not a root of $f(x)=x^2$. 
However, this leaves the possibility that $\log_\cA(\alpha)$ is a root of $f(x)=x(x-r)$ for some $r\in\mathbb{Q}^\times$. This would mean $\log_\cA(\alpha)$ is a rational multiple of a non-trivial idempotent in $\cA$, and ruling out this mixed zero/non-zero case appears to require further ideas.
\end{remark}

\end{document}